\documentclass[12pt]{amsart}

\usepackage{amscd}
\usepackage{verbatim}

\usepackage{amssymb}
\usepackage{pb-diagram,pb-xy}

\usepackage{amssymb}

\usepackage{xy}
\xyoption{all}

\numberwithin{equation}{section}

\newtheorem{theorem}[equation]{Theorem}
\newtheorem{proposition}[equation]{Proposition}
\newtheorem{lemma}[equation]{Lemma}
\newtheorem{corollary}[equation]{Corollary}

\theoremstyle{definition}
\newtheorem{remark}[equation]{Remark}
\newtheorem{example}[equation]{Example}

\renewcommand{\(}{\bigl(}
\renewcommand{\)}{\bigr)}

\newcommand{\tens}{\otimes}
\newcommand{\inv}{^{-1}}

\newcommand{\id}{\mathrm{id}}

\newcommand{\CH}{\operatorname{CH}}
\renewcommand{\Im}{\operatorname{Im}}

\newcommand{\Coker}{\operatorname{Coker}}

\newcommand{\ch}{\operatorname{char}}

\newcommand{\Spec}{\operatorname{Spec}}

\newcommand{\gSL}{\operatorname{\mathbf{SL}}}

\newcommand{\gPGL}{\operatorname{\mathbf{PGL}}}

\newcommand{\End}{\operatorname{End}}
\newcommand{\Hom}{\operatorname{Hom}}

\newcommand{\Nrd}{\operatorname{Nrd}}
\newcommand{\ra}{\rightarrow}
\newcommand{\xra}{\xrightarrow}

\newcommand{\A}{\mathbb{A}}
\renewcommand{\P}{\mathbb{P}}
\newcommand{\Z}{\mathbb{Z}}

\newcommand{\R}{\mathbb{R}}

\newcommand{\gm}{\mathbb{G}_m}

\newcommand{\cD}{\mathcal D}

\newcommand{\cX}{\mathcal X}
\newcommand{\cT}{\mathcal T}

\newcommand{\SB}{\operatorname{S\!\hspace{0.25ex}B}\nolimits}
\newcommand{\Ch}{\mathop{\mathrm{Ch}}\nolimits}

\newcommand{\F}{\mathbb{F}}
\newcommand{\cR}{\mathcal R}
\newcommand{\Char}{\mathop{\mathrm{char}}\nolimits}
\newcommand{\corr}{\rightsquigarrow}
\newcommand{\pr}{\operatorname{\mathit{pr}}}

\usepackage[hypertex]{hyperref}

\title
[Motivic decomposition] 
{Motivic decomposition of compactifications \\
of certain group varieties}


\author
{Nikita A. Karpenko}

\address
{Mathematical \& Statistical Sciences \\
University of Alberta \\
Edmonton
\\
CANADA}

\email
{karpenko {\it at} ualberta.ca}

\author
{Alexander S. Merkurjev}

\address
{Department of Mathematics\\
University of California\\
Los Angeles\\
CA\\
USA}

\email{merkurev {\it at} math.ucla.edu}

\date
{February 14, 2014. Extended: February 22, 2014}

\thanks
{The first author acknowledges a partial support of the French Agence Nationale
de la Recherche (ANR) under reference ANR-12-BL01-0005;
his work has been also supported by the start-up grant of the University of Alberta.
The work of the second author has been supported by the
NSF grant DMS \#1160206 and the Guggenheim Fellowship.}

\begin{document}

\begin{abstract}
Let $D$ be a central simple algebra of prime degree over a field
and let $E$ be an $\gSL_1(D)$-torsor.
We determine the complete motivic decomposition of certain compactifications of $E$.
We also compute the Chow ring of $E$.
\end{abstract}

\maketitle

\tableofcontents

\section
{Introduction}

Let $p$ be a prime number.
For any integer $n\geq2$, a {\em Rost motive of degree $n$} is a direct summand $\cR$ of the Chow motive with coefficients in $\Z_{(p)}$ (the localization of the integers at the prime ideal $(p)$)
of a smooth complete geometrically irreducible variety $X$ over a field $F$ such that for any extension field $K/F$ with a closed point on $X_K$ of degree prime to $p$, the motive $\cR_K$ is isomorphic to the direct sum of Tate motives
$$
\Z_{(p)}\oplus \Z_{(p)}(b)\oplus \Z_{(p)}(2b)\oplus \dots\oplus\Z_{(p)}((p-1)b),
$$
where $b=(p^{n-1}-1)/(p-1)$.
The isomorphism class of $\cR$ is determined by $X$, \cite[Proposition 3.4]{snv};
$\cR$ is indecomposable as long as $X$ has no closed points of degree prime to $p$.

A smooth complete geometrically irreducible variety $X$ over $F$
is a \emph{$p$-generic splitting variety} for an element
$s\in H^n_{\acute{e}t}(F, \Z/p\Z(n-1))$,
if $s$ vanishes over a field extension $K/F$ if and only if $X$ has a
closed point of degree prime to $p$ over $K$.
A \emph{norm variety} of $s$ is a $p$-generic splitting variety of dimension $p^{n-1}-1$.

A Rost motive living on a $p$-generic splitting variety of an element $s$ is determined by $s$ up to isomorphism and called the Rost motive of $s$.
In characteristic $0$, any {\em symbol} $s$ admits a norm variety possessing a Rost motive.
This played an important role in the proof of the Bloch-Kato conjecture (see \cite{MR2811603}).
It is interesting to understand the complement to the Rost motive in the motive of a norm variety $X$ for a given $s$;
this complement, however, depends on $X$ and is not determined by $s$ anymore.

For $p=2$, there are nice norm varieties known as norm quadrics.
Their complete motivic decomposition is a classical result due to M. Rost.
A norm quadric $X$ can be viewed as a compactification of the affine quadric $U$ given by $\pi=c$, where $\pi$ is a quadratic ($n-1$)-fold Pfister form
and $c\in F^\times$.
The summands of the complete motivic decomposition of $X$ are given by the degree $n$ Rost motive of $X$ and shifts of the degree $n-1$
Rost motive of the projective Pfister quadric $\pi=0$. It is proved in \cite[Theorem A.4]{MR1836000} that $\CH(U)=\Z$. In the present paper we extend these results
to arbitrary prime $p$ (and $n=3$).

For arbitrary $p$, there are nice norm varieties in small degrees.
For $n=2$, these are the Severi-Brauer varieties of degree $p$ central simple $F$-algebras.
Any of them admits a degree $2$ Rost motive which is simply the total motive of the variety.

The first interesting situation occurs in degree $n=3$.
 Let $D$ be a degree $p$ central division $F$-algebra, $G=\gSL_1(D)$ the special linear group of $D$, and $E$ a principle homogeneous space under $G$.
 The affine variety $E$ is given by the equation $\Nrd=c$, where $\Nrd$ is the reduced norm of $D$ and $c\in F^\times$.
 Any smooth compactification of $E$ is a norm variety of the element
 $s:=[D]\cup(c)\in H_{\acute{e}t}^3(F,\Z/p\Z(2))$.
It has been shown by N. Semenov in \cite{MR2400993} for $p=3$ (and $\Char F=0$) that the motive of a certain smooth equivariant compactification of $E$ decomposes in a direct sum, where one of the summands is the Rost motive  of $s$, another summand is a motive $\varepsilon$ vanishing over any field extension of $F$ splitting $D$, and each of the remaining summands is a shift of the motive of the Severi-Brauer variety of $D$.
All these summands (but $\varepsilon$) are indecomposable and $\varepsilon$ was expected to be $0$.

Another proof of this result (covering arbitrary characteristic) has been provided in \cite{MR2393083} along with the claim that $\varepsilon=0$, but the proof of the claim was incomplete.

In the present paper we prove the following main result (see {Theorem \ref{cor1}}):


\begin{theorem}\label{really main}
Let $F$ be a field, $D$ a central division $F$-algebra of prime degree $p$, $X$ a smooth compactification of an $\gSL_1(D)$-torsor, and $M(X)$ its Chow motive with $\Z_{(p)}$-coefficients.
Assume that $M(X)$ over the function field of the Severi-Brauer variety $S$ of $D$ is isomorphic to a direct sum of Tate motives.
Then $M(X)$ (over $F$) is isomorphic to the direct sum of the Rost motive of $X$ and several shifts of $M(S)$.
This is the unique decomposition of $M(X)$ into a direct sum of indecomposable motives.
\end{theorem}

We note that the compactification in \cite{MR2400993} (for $p=3$) has the property required in Theorem \ref{really main} (see Example \ref{semenovex}).

In Section \ref{compact} we show that the condition that $M(X)$ is split over $F(S)$ is satisfied for all smooth $G\times G$-equivariant compactifications of $G=\gSL_1(D)$.
Moreover, we prove that the motive $M(X)$ is split for all smooth equivariant compactifications $X$ of split semisimple groups (see Theorem \ref{equivsplit}).

 We also compute the Chow ring of $G$ in arbitrary characteristic as well as the Chow ring of $E$ in characteristic $0$ (see Theorem \ref{chow2} and Corollary \ref{cor2}):

 \begin{theorem}\label{chow1}
Let $D$ be a central division algebra of prime degree $p$ and $G=\gSL_1(D)$.

\noindent {\rm 1)} There is an element $h\in\CH^{p+1}(G)$ such  that
  \newcommand{\xqedhere}[2]{%
  \rlap{\hbox to#1{\hfil\llap{\ensuremath{#2}}}}}
\[
\CH(G)=\Z\cdot 1\oplus (\Z/p\Z)h\oplus (\Z/p\Z)h^2\oplus\cdots\oplus (\Z/p\Z)h^{p-1}.
\]

\noindent {\rm 2)} Let $E$ be a nonsplit $G$-torsor. If $\Char F=0$, then $\CH(E)=\Z$.
\end{theorem}

\smallskip
\noindent
{\sc Acknowledgements.}
We thank Michel Brion for teaching us the theory of equivariant compactifications. We also thank
Markus Rost and Kirill Zainoulline
for helpful information.

\section{$K$-cohomology}

Let $X$ be a smooth variety over $F$. We write $A^i(X,K_n)$ for the $K$-cohomology groups as defined in \cite{Rost98a}.
In particular, $A^i(X,K_i)$ is the Chow group $\CH^i(X)$ of classes of codimension $i$ algebraic cycles on $X$.

Let $G$ be a simply connected semisimple algebraic group. The group $A^1(G,K_2)$ is additive in $G$, i.e., if $G$ and $G'$ are two simply connected group,
then the projections of $G\times G'$ onto $G$ and $G'$ yield an isomorphism (see \cite[Part II, Proposition 7.6 and Theorem 9.3]{GMS})
\[
A^1(G,K_2)\oplus A^1(G',K_2)\stackrel{\sim}{\longrightarrow} A^1(G\times G',K_2).
\]
The following lemma readily follows.

\begin{lemma}\label{addit}
1) The map
\[
A^1(G,K_2)\to  A^1(G\times G,K_2)=A^1(G,K_2)\oplus A^1(G,K_2)
\]
induced by the product homomorphism $G\times G\to G$ is equal to $(1,1)$.

2) The map $A^1(G,K_2)\to A^1(G,K_2)$ induced
by the morphism $G\to G$, $x\mapsto x\inv$ is equal to $-1$.
\end{lemma}

\begin{proof}
1)
It suffices to note that the isomoprhism
$$
A^1(G\times G',K_2){\stackrel{\sim}{\longrightarrow}}A^1(G,K_2)\oplus A^1(G',K_2)
$$
inverse to the one mentioned above, is given by the pull-backs with respect to the group embeddings
$G,G'\hookrightarrow G\times G'$.

2)
The composition of the embedding of varieties $G\hookrightarrow G\times G$, $g\mapsto (g,g^{-1})$ with the product map
$G\times G\to G$ is trivial.
\end{proof}

If $G$ is an absolutely simple simply connected group, then $A^1(G,K_2)$ is an infinite cyclic group with a canonical
generator $q_G$ (see \cite[Part II, \S 7]{GMS}).

\section{BGQ spectral sequence}

Let $X$ be a smooth variety over $F$.
We consider the Brown-Gersten-Quillen \emph{coniveau spectral sequence}
\begin{equation}\label{bgq}
E_2^{s,t}=A^s(X,K_{-t})\Rightarrow K_{-s-t}(X)
\end{equation}
converging to the $K$-groups of $X$ with the topological filtration
\cite[\S 7, Th. 5.4]{Quillen73}.

\begin{example}\label{split}
Let $G=\gSL_n$. By \cite[\S 2]{Suslin91a}, we have $\CH(G)=\Z$.
It follows that all the differentials of the BGQ
spectral sequence for $G$ coming to the zero diagonal are trivial.
\end{example}

\begin{lemma}[{\cite[Theorem 3.4]{Merkurjev10}}]
\label{div}
If $\delta$ is a nontrivial differential in the spectral sequence (\ref{bgq}) on the $q$-th page $E_q^{*,*}$,
then $\delta$ is of finite order and for every prime divisor $p$ of the order of $\delta$, the integer $p-1$ divides $q-1$.
\end{lemma}

Let $p$ be a prime integer, $D$ a central division algebra over $F$ of degree $p$ and $G=\gSL_1(D)$. As $D$ is split by a field
extension of degree $p$, it follows from Example \ref{split} that all Chow groups $\CH^i(G)$ are $p$-periodic for $i>0$ and
the order of every differential in the BGQ
spectral sequence for $G$ coming to the zero diagonal divides $p$.
The edge homomorphism $K_1(G)\to E_2^{0,-1}=A^0(G, K_1)=F^\times$ is a surjection
split by the pull-back with respect to the structure morphism $G\to\Spec F$.
Therefore, all the
differentials starting at $E_*^{0,-1}$ are trivial.

It follows then from Lemma \ref{div} that the
only possibly nontrivial differential coming to the terms $E_q^{i,-i}$ for $q\geq2$ and $i\leq p+1$ is
\[
\partial_G: A^1(G,K_2)=E_p^{1,-2}\to E_p^{p+1,-p-1}=\CH^{p+1}(G).
\]
By \cite[Theorem 6.1]{Suslin91a} (see also \cite[Theorem 5.1]{MR1618404}), $K_0(G)=\Z$, hence the factors
$$
K_0(G)^{(i)}/K_0(G)^{(i+1)}=E_\infty^{i,-i}
$$
of the topological filtration on $K_0(G)$
are trivial for $i>0$. It follows that the map $\partial_G$ is surjective. As the group $A^1(G,K_2)$ is cyclic with the generator $q_G$,
the group $\CH^{p+1}(G)$ is cyclic of order dividing $p$. It is shown in \cite[Theorem 4.2]{Yagunov07} that the differential $\partial_G$
is nontrivial. We have proved the following lemma.

\begin{lemma}\label{p+1}
If $D$ is a central division algebra, then $\CH^{p+1}(G)$ is a cyclic group of order $p$ generated by $\partial_G(q_G)$.
\qed
\end{lemma}

\section{Specialization}\label{spec}

Let $A$ be a discrete valuation ring with residue field $F$ and quotient field $L$. Let $\cX$ be a smooth scheme over $A$ and
set $X=\cX\tens_A F$, $X'=\cX\tens_A L$. By \cite[Example 20.3.1]{Fulton84}, there is a \emph{specialization} ring homomorphism
\[
\sigma:\CH^*(X')\to \CH^*(X).
\]

\begin{example}
Let $X$ be a variety over $F$, $L=F(t)$ the rational function field. Consider the valuation ring $A\subset L$ of the parameter $t$
and $\cX=X\tens_F A$. Then $X'=X_L$ and we have a specialization ring homomorphism $\sigma:\CH^*(X_L)\to \CH^*(X)$.
\end{example}

A section of the structure morphism $\cX\to \Spec A$ gives two rational points $x\in X$ and $x'\in X'$. By definition of the specialization,
$\sigma([x'])=[x]$.

Let $F$ be a field of finite characteristic. By \cite[Ch. IX, \S 2, Propositions 5 and 1]{Bourbaki83}, there is a complete
discretely valued field $L$ of characteristic zero with residue field $F$. Let $A$ be the valuation ring and $D$ a central simple algebra
over $F$. By \cite[Theorem 6.1]{Grothendieck68}, there is an Azumaya algebra $\cD$ over $A$ such that
$D\simeq \cD\tens_A F$. The algebra $D'=\cD\tens_A L$ is a central simple algebra over $L$. Then we have a specialization homomorphism
\[
\sigma: \CH^*(\gSL_1(D'))\to \CH^*(\gSL_1(D))
\]
satisfying $\sigma([e'])=[e]$, where $e$ and $e'$ are the identities of the groups.

\section{A source of split motives}

We work in the category of Chow motives over a field $F$, \cite[\S64]{EKM}. We write $M(X)$ for the motive (with integral coefficients) of a smooth
complete variety $X$ over $F$.

A motive is {\em split} if it is isomorphic to a finite direct sum of Tate motives $\Z(a)$ (with arbitrary shifts $a$).
Let $X$ be a smooth proper variety such that the motive $M(X)$ is split, i.e., $M(X)=\coprod_i \Z(a_i)$ for some $a_i$.
The \emph{generating} (\emph{Poincar\' e}) polynomial
$P_X(t)$ of $X$ is defined by
\[
P_X(t)=\sum_i t^{a_i}.
\]
Note that the integer $a_i$ is equal to the rank of the (free abelian) Chow group $\CH^i(X)$.

\begin{example}\label{flags}
Let $G$ be a split semisimple group and $B\subset G$ a Borel subgroup. Then
\[
P_{G/B}(t)=\sum_{w\in W} t^{l(w)},
\]
where $W$ is the Weyl group of $G$ and $l(w)$ is the length of $w$ (see \cite[\S3]{MR0354697}).
\end{example}

\begin{proposition}[{P.~Brosnan, \cite[Theorem 3.3]{Brosnan05}}]
\label{brosnan}
Let $X$ be a smooth projective variety over $F$ equipped
with an action of the multiplicative group $\gm$. Then
\[
M(X) =\coprod_i M(Z_i)(a_i),
\]
where the $Z_i$ are the (smooth) connected components of the subscheme of $X^{\gm}$
of fixed points and $a_i\in\Z$. Moreover, the integer $a_i$ is the dimension
of the positive eigenspace of the action of $\gm$ on the tangent space $\cT_z$
of $X$ at an arbitrary closed point $z\in Z_i$. The dimension of $Z_i$ is the
dimension of $(\cT_z)^{\gm}$.
\end{proposition}

Let $T$ be a split torus of dimension $n$.
The choice of a $\Z$-basis in the character group $T^*$ allows us to identify $T^*$ with $\Z^n$. We
order $\Z^n$ (and hence $T^*$) lexicographically.

Suppose
$T$ acts on a smooth variety $X$ and let $x\in X$ be an $T$-fixed rational point.
Let $\chi_1, \chi_2,\dots, \chi_m$ be all characters of the representation of $T$ in the tangent space $\cT_x$ of $X$ at $x$.
Write $a_x$ for the number of positive (with respect to the ordering) characters among the $\chi_i$'s.

\begin{corollary}\label{totalsplit}
Let $X$ be a smooth projective variety over $F$ equipped
with an action of a split torus $T$. If the subscheme $X^T$ of $T$-fixed points in $X$ is a
disjoint union of finitely many rational points,
the motive of $X$ is split. Moreover,
\[
P_X(t)= \sum_{x\in X^T} t^{a_x}.
\]
\end{corollary}

\begin{proof}
Induction on the dimension of $T$.
\end{proof}

\begin{example}\label{toricex}
Let $T$ be a split torus of dimension $n$ and $X$ a smooth projective toric variety (see \cite{Fulton93}).
Let $\sigma$ be a cone of dimension $n$ in the fan of $X$ and $\{\chi_1, \chi_2,\dots,\chi_n\}$ a
(unique) $\Z$-basis of $T^*$ generating the dual cone $\sigma^\vee$. The standard $T$-invariant
affine open set corresponding to $\sigma$ is $V_\sigma:=\Spec F[\sigma^\vee]$. The map $V_\sigma\to \A^n$,
taking $x$ to $(\chi_1(x),\chi_2(x),\dots,\chi_n(x))$ is a $T$-equivariant isomorphism, where $t\in T$ acts
on the affine space $\A^n$ by componentwise multiplication by $\chi_i(t)$. The only one $T$-equivariant point $x\in V_\sigma$
corresponds to the origin under the isomorphism, so we can identify the tangent space $\cT_x$ with $\A^n$,
and the $\chi_i$'s are the characters of the representation of $T$ in the tangent space $\cT_x$.
Let $a_\sigma$ be the number of positive $\chi_i$'s with respect to a fixed lexicographic order on $T^*$.
Every $T$-fixed point in $X$ belongs to $V_\sigma$ for a unique $\sigma$. It follows that the motive $M(X)$ is split and
\[
P_X(t)=\sum_{\sigma} t^{a_\sigma},
\]
where the sum is taken over all dimension $n$ cones in the fan of $X$.
\end{example}

\section{Compactifications of algebraic groups}\label{compact}

A {\em compactification} of an affine algebraic group $G$ is a projective variety containing $G$ as a dense open subvariety.
A $G\times G$-\emph{equivariant} compactification of $G$ is a projective variety $X$ equipped with an action of $G\times G$ and containing the homogeneous variety
$G=(G\times G)/\operatorname{diag}(G)$
as an open orbit.
Here the group $G\times G$ acts on $G$ by the left-right translations.

Let $G$ be a split semisimple group over $F$. Write $G_{ad}$ for the corresponding adjoint group. The group $G_{ad}$
admits the so-called \emph{wonderful} $G_{ad}\times G_{ad}$-equivariant compactification $\bf X$ (see \cite[\S 6.1]{BK05}).
Let $T\subset G$ be a split maximal torus and $T_{ad}$ the corresponding maximal torus in $G_{ad}$.
The closure $\bf X'$ of $T_{ad}$ in $\bf X$ is a toric $T_{ad}$-variety with fan consisting of all Weyl chambers in $(T_{ad})_*\tens\R=T_*\tens\R$
and their faces.

Let $B$ be a Borel subgroup of $G$ containing $T$ and $B^-$ the opposite Borel subgroup. There is an open $B^-\times B$-invariant subscheme ${\bf X}_0\subset \bf X$ such
that the intersection ${\bf X}'_0:={\bf X}_0\cap  \bf X'$ is the standard open $T_{ad}$-invariant subscheme of the toric variety $\bf X'$ corresponding to the
negative Weyl chamber $\Omega$ that is a cone in the fan of $\bf X'$. Note that the Weyl group $W$ of $G$ acts simply transitively on the set of
all Weyl chambers.

A $G\times G$-equivariant compactification $X$ of $G$ is called \emph{toroidal} if $X$ is normal and the quotient map $G\to G_{ad}$
extends to a morphism $\pi:X\to \bf X$ (see \cite[\S 6.2]{BK05}). The closed subscheme $X':=\pi\inv(\bf X')$ of $X$ is a projective toric $T$-variety.
Note that the diagonal subtorus $\operatorname{diag}(T)\subset T\times T$ acts trivially on $X'$.
The fan of $X'$ is a subdivision of the fan consisting of the Weyl chambers
and their faces. The scheme $X$ is smooth if and only if so is $X'$.

Conversely, if $F$ is a perfect field, given a smooth projective toric $T$-variety with a $W$-invariant fan that is a subdivision of the fan consisting of the Weyl chambers
and their faces, there is a unique smooth $G\times G$-equivariant toroidal compactification $X$ of $G$
with the toric variety $X'$ isomorphic to the given one (see \cite[\S 6.2]{BK05} and \cite[\S 2.3]{Huruguen11}).
By \cite{Brylinski79} and \cite{CTHS05}, such a smooth toric variety exists for every split semisimple group $G$. In other words, the following holds.

\begin{proposition}\label{torexist}
Every split semisimple group $G$ over a perfect field admits a smooth $G\times G$-equivariant toroidal compactification. \qed
\end{proposition}

Let $X$ be a smooth $G\times G$-equivariant toroidal compactification of $G$ over $F$. Recall that the toric $T$-variety $X'$ is smooth projective.
Set $X_0:=\pi\inv({\bf X}_0)$ and $X'_0:=\pi\inv({\bf X}'_0)=X'\cap X_0$. Then the $T$-invariant subset $X'_0\subset X'$ is the union of standard open subschemes $V_\sigma$
of $X'$ (see Example \ref{toricex}) corresponding to all cones $\sigma$ in the negative Weyl chamber $\Omega$. The subscheme $(V_\sigma)^T$ reduces to a single rational point
if $\sigma$ is of largest dimension.
In particular, the subscheme $(X'_0)^T$ of $T$-fixed points in $X'_0$ is a disjoint union of $k$ rational points, where $k$ is the number of cones
of maximal dimension in $\Omega$. It follows that $|(X')^T|=k|W|$, the number of all cones of maximal dimension in the fan of $X'$.

Let $U$ and $U^-$ be the unipotent radicals of $B$ and $B^-$ respectively.

\begin{lemma}[{\cite[Proposition 6.2.3]{BK05}}]\label{brion}

\noindent {\rm 1)} Every $G\times G$-orbit in $X$ meets $X'_0$ along a unique $T$-orbit.

\noindent {\rm 2)} The map
\[
U^-\times X'_0 \times U\to X_0,\quad (u,x,v)\mapsto uxv\inv,
\]
is a $T\times T$-equivariant isomorphism.

\noindent {\rm 3)} Every closed $G\times G$-orbit in $X$ is isomorphic to $G/B\times G/B$.
\end{lemma}

\begin{proposition}\label{fixedp}
The scheme $X^{T\times T}$ is the disjoint union of $Wx_0W$ over all $x_0\in (X'_0)^T$ and $Wx_0W$ is a disjoint union of $|W|^2$
rational points.
\end{proposition}

\begin{proof}
Take $x\in X^{T\times T}$. Let $\bf x$ be the image of $x$ under the map $\pi:X\to \bf X$. Computing dimensions of maximal tori of
the stabilizers of points in the wonderful compactification $\bf X$, we see that $\bf x$ lies in the only closed $G \times G$-orbit
${\bf O}$ in ${\bf X}$ (e.g., \cite[Lemma 4.2]{EF08}). By Lemma \ref{brion}(3), applied to the compactification ${\bf X}$ of
$G_{ad}$, ${\bf O}\simeq G/B\times G/B$.
In view of Lemma \ref{brion}(1), ${\bf O}\cap {\bf X}'_0$ is a closed $T$-orbit in ${\bf X}'_0$ and therefore, reduces to a single rational
$T$-invariant point in ${\bf X}'_0$. The group $W\times W$ acts simply transitively on the set of $T\times T$-fixed point in $G/B\times G/B$.
It follows that $|W{\bf x}W|=|W|^2$ and $W{\bf x}W$ intersects ${\bf X}'_0$. Therefore, $WxW$ intersects $X^{T\times T}\cap X'_0=(X'_0)^T$, that is
the disjoint union of $k$ rational points. Hence $x$ is a rational point, $x\in W(X'_0)^T W$ and $|WxW|=|W|^2$.

Note that for a point $x_0\in (X'_0)^T$, the $G\times G$-orbit of $x_0$ intersects $X'_0$ by the $T$-orbit $\{x_0\}$ in view of Lemma \ref{brion}(1).
It follows that different $Wx_0W$ do not intersect and therefore, $X^{T\times T}$ is the disjoint union of $Wx_0W$ over all $x_0\in (X'_0)^T$.
\end{proof}

Let $X$ be a smooth $G\times G$-equivariant toroidal compactification of a split semisimple group $G$ of rank $n$.
By Proposition \ref{fixedp}, every $T\times T$-fixed point $x$ in $X$ is of the form $x=w_1x_0w_2\inv$, where
$w_1,w_2\in W$ and $x_0\in (X'_0)^T$. Recall that $X'_0$ is the union of the standard affine open subsets $V_\sigma$
of the toric $T$-variety $X'$ over all cones $\sigma$ of dimension $n$ in the Weyl chamber $\Omega$. Let $\sigma$
be a (unique) cone in $\Omega$ such that $x_0\in V_\sigma$.

By Lemma \ref{brion}(2), the map
\[
f:U^-\times V_\sigma \times U\to X,\quad (u_1,y,u_2)\mapsto w_1u_1 x_0 u_2\inv w_2\inv
\]
is an open embedding. We have $f(1,x_0,1)=x$. Thus, $f$ identifies the tangent space $\cT_x$ of $x$ in $X$ with the
space $\frak{u^-}\oplus \frak{a}\oplus \frak{u}$, where $\frak{u}$ and $\frak{u^-}$ are the Lie algebras of $U$ and $U^-$ respectively
and $\frak{a}$ is the tangent space of $V_\sigma$ at $x'$. The torus $T\times T$ acts linearly on the tangent space $\cT_x$
leaving invariant $\frak{u^-}$, $\frak{a}$ and $\frak{u}$. For
convenience, we write $T\times T$ as $S:=T_1\times T_2$ in order to distinguish the components.
Let $\Phi_1^-$ and $\Phi_2^-$ be two copies of the set of negative roots in $T_1^*$ and $T_2^*$ respectively.
The set of characters of
the $S$-representation $\frak{u^-}$ (respectively, $\frak{u}$) is ${w_1}(\Phi_1^-)$ (respectively, ${w_2}(\Phi_2^-)$).

Let $\{\chi_1, \chi_2,\dots,\chi_n\}$ be a
(unique) $\Z$-basis of $T^*$ generating the dual cone $\sigma^\vee$. By Example \ref{toricex}, the set of characters of
the $S$-representation $\frak{a}$ is
\[
\{(w_1(\chi_i), - w_2(\chi_i))\}_{i=1}^{n}\subset S^*=T_1^*\oplus T_2^*.
\]
Let $\Pi_1$ and $\Pi_2$ be (ordered) systems of simple roots in $\Phi_1$ and $\Phi_2$ respectively. Consider the lexicographic ordering
on $S^*=T_1^*\oplus T_2^*$ corresponding to the basis $\Pi_1\cup \Pi_2$ of $S^*$. As $\chi_i\neq 0$, we have $(w_1(\chi_i), - w_2(\chi_i))>0$
if and only if $w_1(\chi_i)>0$. For every $w\in W$, write $b(\sigma, w)$ for the number of all $i$ such that $w(\chi_i)> 0$.
Note that the number of positive roots in $w(\Phi^-)$ is equal to the length $l(w)$ of $w$.
By Corollary \ref{totalsplit}, we have
\begin{equation}\label{mainformula}
P_X(t)=\sum_{w_1,w_2\in W,\ \sigma\subset\Omega} t^{l(w_1)+b(\sigma, w_1)+l(w_2)}=\Big(\sum_{w\in W,\ \sigma\subset\Omega} t^{l(w)+b(\sigma, w)}\Big)\cdot P_{G/B}(t),
\end{equation}
as by Example \ref{flags},
\[
P_{G/B}(t)=\sum_{w\in W} t^{l(w)}.
\]

We have proved the following theorem.

\begin{theorem}\label{semis}
Let $X$ be a smooth $G\times G$-equivariant toroidal compactification of a split semisimple group $G$. Then the motive $M(X)$ is split into a direct sum of
$s|W|$ Tate motives, where $s$ is the number of cones of maximal dimension in the fan of the associated toric variety $X'$. Moreover,
\[
P_X(t)=\Big(\sum_{w\in W,\ \sigma\subset\Omega} t^{l(w)+b(\sigma, w)}\Big)\cdot P_{G/B}(t).
\]
In particular, the motive $M(X)$ is divisible by $M(G/B)$.
\end{theorem}

\begin{example}
Let $G$ be a semisimple adjoint group and $X$ the wonderful compactification of $G$. Then the negative Weyl chamber $\Omega$ is the cone $\sigma=\Omega$
in the fan of $X'$.
The dual cone $\sigma^\vee$ is generated by $-\Pi$. Hence $b(w,\sigma)$ is equal to the number of \emph{simple} roots $\alpha$ such that $w(\alpha)\in \Phi^-$.
\end{example}

\begin{example}\label{sl2}
Let $G=\gSL_3$, $\Pi=\{\alpha_1,\alpha_2\}$. Bisecting each of the six Weyl chambers we get a smooth projective fan with $12$ two-dimensional cones.
The two cones dual to the ones in the negative Weyl chamber are generated by $\{-\alpha_1,(\alpha_1-\alpha_2)/3\}$ and $\{-\alpha_2,(\alpha_2-\alpha_1)/3\}$ respectively.
Let $X$ be the corresponding $G\times G$-equivariant toroidal compactification of $G$. By (\ref{mainformula}),
\[
P_X(t)=(t^5+t^4+4t^3+4t^2+t+1)(t^3+2t^2+2t+1).
\]
\end{example}

Now consider arbitrary (not necessarily toroidal) $G\times G$-equivariant compactifications.

\begin{theorem} \label{equivsplit}
Let $X$ be a smooth $G\times G$-equivariant compactification of a split semisimple group $G$ over $F$. Then the subscheme $X^{T\times T}$
is a disjoint union of finitely many rational points. In particular, the motive $M(X)$ is split.
\end{theorem}

\begin{proof}
By \cite[Proposition 6.2.5]{BK05}, there is a $G\times G$-equivariant toroidal compactification $\widetilde X$ of $G$ together with
a $G\times G$-equivariant morphism $\varphi:\widetilde X\to X$. Let $x\in X^{T\times T}$. By Borel's fixed point theorem, the fiber $\varphi\inv(x)$
has a $T\times T$-fixed point, so the map $\widetilde X^{T\times T}\to X^{T\times T}$ is surjective. By Proposition \ref{fixedp}, $\widetilde X^{T\times T}$
is a disjoint union of finitely many rational points, hence so is $X^{T\times T}$.
\end{proof}

\begin{example}\label{twist}
Let $Y$ be a smooth $H\times H$-equivariant compactification of the group $H=\gSL_n$ over $F$. In particular the projective linear group
$\gPGL_n$ acts on $Y$ by conjugation. Let $D$ be a central simple $F$-algebra of degree $n$ and $J$ the corresponding $\gPGL_n$-torsor.
The twist of $H$ by $J$ is the group $G=\gSL_1(D)$, hence the twist $X$ of $Y$ is a smooth $G\times G$-equivariant compactification of $G$.
If $E$ is a $G$-torsor, one
can twist $X$ by $E$ to get a smooth compactification of $E$.
By Theorem \ref{equivsplit}, the motives of these compactifications are split over
every splitting field of $D$.
\end{example}


\section{Some computations in $\CH(\gSL_1(D))$}

Let $D$ be a central simple algebra of prime degree $p$ over $F$ and $G=\gSL_1(D)$.

\begin{lemma}\label{infinity}
Let $X$ be a smooth compactification of $G$. Then $D$ is split by the residue field of every point
in $X\setminus G$.
\end{lemma}

\begin{proof}
Let $Y$ be the projective (singular) hypersurface given in the projective space $\P(D\oplus F)$ by the equation $\Nrd=t^p$,
where $\Nrd$ is the reduced norm form. The group $G$ is an open subset in $Y$,
so we can identify the function fields $F(X)=F(G)=F(Y)$. Let $x\in X\setminus G$. As $x$ is smooth in $X$,
there is a regular system of local parameters around $x$ and therefore a
valuation $v$ of $F(G)$ over $F$ with residue field $F(x)$. Since $Y$ is projective, $v$ dominates a point
$y\in Y\setminus G$. Over the residue field $F(y)$ the norm form $\Nrd$ is isotropic, hence $D$ is split over $F(y)$.
Since $v$ dominates $y$, the field $F(y)$ is contained in $F(v)=F(x)$. Therefore, $D$ is split over $F(x)$.
\end{proof}

\begin{lemma}\label{dimzero}
If $D$ is a division algebra, then the group $\CH_0(G)=\CH^{p^2-1}(G)$ is cyclic of order $p$
generated by the class of the identity $e$ of $G$.
\end{lemma}

\begin{proof}
The group of $R$-equivalence classes of points in $G(F)$ is equal to $\operatorname{SK}_1(D)$ (see \cite[Ch. 6]{Voskresenski98}) and hence is trivial
by a theorem of Wang. It follows that we have $[x]=[e]$ in $\CH_0(G)$ for every rational point $x\in G(F)$.
If $x\in G$ is a closed point, then $[x']=[e]$ in $\CH_0(G_K)$, where $K=F(x)$ and $x'$ is a rational point of $G_K$ over $x$.
Taking the norm homomorphism $\CH_0(G_K)\to \CH_0(G)$ for the finite field extension $K/F$, we have $[x]=[K:F]\cdot [e]$
in $\CH_0(G)$. It follows that $\CH_0(G)$ is a cyclic group generated by $[e]$.

As $p\cdot \CH_0(G)=0$ it suffices to show that $[e]\neq 0$ in $\CH_0(G)$. Let $Y$ be the compactification of $G$ as in
the proof of Lemma \ref{infinity} and let $Z=Y\setminus G$. As $D$ is a central division algebra, the degree of every closed point of $Z$ is divisible by $p$ by Lemma \ref{infinity}.
It follows that the class $[e]$ in $\CH_0(Y)$ does not belong to the image of the push-forward homomorphism $i$ in the exact sequence
\[
\CH_0(Z)\xra{i}\CH_0(Y)\to \CH_0(G)\to 0.
\]
Therefore, $[e]\neq 0$ in $\CH_0(G)$.
\end{proof}

Consider the morphism $s:G\times G\to G$, $s(x,y)=xy\inv$.
Note that $s$ is flat as the composition of the automorphism
$(x,y)\mapsto(xy^{-1},y)$ of the variety $G\times G$ with the projection
$G\times G\to G$.

Let $h=\partial_G(q_G) \in\CH^{p+1}(G)$.

\begin{lemma}\label{pull}
We have $s^*(h)=h\times 1 - 1\times h$ in $\CH^{p+1}(G\times G)$.
\end{lemma}

\begin{proof}
By Lemma \ref{addit}, we have $s^*(q_G)=q_G\times 1 -1\times q_G$ in $A^1(G\times G,K_2)$.
The differentials $\partial_G$ commute with flat pull-back maps, hence we have
\[
s^*(h)=s^*(\partial_G(q_G))=\partial_{G\times G}(s^*(q_G))=\partial_{G\times G}(q_G\times 1-1\times q_G)=
\]
\[\partial_G(q_G)\times 1 - 1\times \partial_G(q_G)=h\times 1 - 1\times h. \qedhere
\]
\end{proof}

\begin{proposition}\label{diagonal}
Let $c$ be an integer with $h^{p-1}=c[e]$ in $\CH^{p^2-1}(G)$. Then
\[
c\Delta_G=\sum_{i=0}^{p-1}h^i\times h^{p-1-i},
\]
where $\Delta_G$ is the class of the diagonal $\operatorname{diag}(G)$ in $\CH^{p^2-1}(G\times G)$.
\end{proposition}

\begin{proof}
The diagonal in $G\times G$ is the pre-image of $e$ under $s$. Hence by Lemma \ref{pull},
\[
c\Delta_G=cs^*([e])=s^*(h^{p-1})=(h\times 1-1\times h)^{p-1}=\sum_{i=0}^{p-1}h^i\times h^{p-1-i}
\]
as $\binom{p-1}{i}\equiv (-1)^i$ modulo $p$ and
$ph=0$.
\end{proof}

\section{Rost's theorem}\label{rost}

We have proved in Lemma \ref{p+1} that if $D$ is a central division algebra, then $\partial_G(q_G)\neq 0$ in $\CH^{p+1}(G)$.
This result is strengthened in Theorem \ref{power} below.

\begin{lemma}\label{chut}
If there is an element $h\in\CH^{p+1}(G)$ such that
$h^{p-1}\neq 0$, then $\partial_G(q_G)^{p-1}\neq 0$.
\end{lemma}

\begin{proof}
By Lemma \ref{p+1},
$h$ is a multiple of $\partial_G(q_G)$.
\end{proof}

\begin{theorem}[M. Rost]\label{rostth}
\label{power}
Let $D$ be a central division algebra of degree $p$, $G=\gSL_1(D)$. Then $\partial_G(q_G)^{p-1}\neq 0$ in $\CH^{p^2-1}(G)=\CH_0(G)$.
\end{theorem}

\begin{proof}

\noindent \emph{Case 1}: Assume first that $\ch(F)=0$, $F$ contains a primitive $p$-th root of unity and
$D$ is a cyclic algebra, i.e., $D=(a,b)_F$ for some $a,b\in F^\times$.

Let $c\in F^\times$ be an element  such that the symbol
$$
u:=(a,b,c)\in H^3_{\acute{e}t}(F,\Z/p\Z(3))\simeq H^3_{\acute{e}t}(F,\Z/p\Z(2))
$$
is nontrivial modulo $p$.
Consider a norm variety $X$ of $u$.

Then $u$ defines a \emph{basic correspondence} in the cokernel of the homomorphism
$$
\CH^{p+1}(X)\to\CH^{p+1}(X\times X)
$$
given by the difference of the pull-backs with respect to the projections.
A representative in $\CH^{p+1}(X\times X)$ of the basic correspondence is a {\em special correspondence}.
Let $z \in \CH^{p+1}( X_{F(X)} )$ be its pull-back.
The modulo $p$ degree
\[
c(X): = \deg ( z^{p-1} ) \in\Z/p\Z
\]
is independent of the choice of the special correspondence.
The construction of $c(X)$ is natural with respect to morphisms of norm varieties (see \cite{Rost06}).

It is shown in \cite{Rost06} that there is an $X$ such that $c(X)\neq 0$. We claim that
$c(X')\neq 0$ for every norm variety $X'$ of $u$.
As $F(X')$ splits $u$ and
$X$ is $p$-generic, $X$ has a closed point over $F(X')$ of degree prime to $p$, or equivalently, there is
a prime correspondence $X'\corr X$ of multiplicity prime to $p$.
Resolving singularities,
we get a smooth complete variety $X''$ together with the two morphisms
$f:X''\to X$ of degree prime to $p$ and $g:X''\to X'$.
It follows by \cite[Corollary 1.19]{MR2220090} that $X''$ is a norm variety of $u$.
Moreover, $c(X'')=\deg(f)c(X)\neq 0$ in $\Z/p\Z$. As  $c(X'')=\deg(g)c(X')$,
$c(X')$ is also nonzero. The claim is proved.

Let $X$ be a smooth compactification of the $G$-torsor $E$ given by the equation $\Nrd=t$ over the rational function field $L=F(t)$
given by a variable $t$. By the above, since $\{a,b,t\}\ne0$,
we have an element $z \in \CH^{p+1}( X_{L(X)} )$ such that $\deg ( z^{p-1} )\neq 0$
in $\Z/p\Z$. The torsor $E$ is trivial over $L(X)$, i.e. $E_{L(X)}\simeq G_{L(X)}$. Then the restriction of $z$ to the
torsor gives an element $y\in \CH^{p+1}(G_{L(X)})$ with $y^{p-1}\neq 0$. The field extension $L(X)/F$
is purely transcendental. By Section \ref{spec} and Lemma \ref{dimzero}, every specialization homomorphism $\sigma: \CH^{p^2-1}(G_{L(X)})\to \CH^{p^2-1}(G)$ is an isomorphism
taking the class of the identity to the class of the identity.
Specializing, we get an element $h\in \CH^{p+1}(G)$ with $h^{p-1}\neq 0$.
It follows from Lemma \ref{chut} that $\partial_G(q_G)^{p-1}\neq 0$.

\medskip

\noindent \emph{Case 2}: Suppose that $\ch(F)=0$ but $F$ may not contain $p$-th roots of unity and $D$ is an arbitrary division algebra
of degree $p$ (not necessarily cyclic). There is a finite field extension $K/F$ of degree prime to $p$ containing
a primitive $p$-th root of unity and such that the algebra $D\tens_F K$ is cyclic (and still nonsplit). By Case 1, $\partial_G(q_{G})^{p-1}_K\neq 0$
over $K$.
Therefore $\partial_G(q_G)^{p-1}\neq 0$.

\medskip

\noindent \emph{Case 3}:  $F$ is an arbitrary field. Choose a field $L$ of characteristic zero and a central simple algebra $D'$
of degree $p$ over $L$ as in Section \ref{spec} and let $G'=\gSL_1(D')$. By Case 2, there is an element $h'\in \CH^{p+1}(G')$
such that $(h')^{p-1}\neq 0$. Applying a specialization $\sigma$ (see Section \ref{spec}), we have $h^{p-1}\neq 0$ for $h=\sigma(h')$.
By Lemma \ref{chut} again, $\partial_G(q_G)^{p-1}\neq 0$.
\end{proof}

Let $D$ be a central division algebra of degree $p$ over $F$ and $X$ a smooth compactification of $G$.
Let $\bar h\in \CH^{p+1}(X)$ be an element such that $\bar h|_{G}=\partial_G(q_G)\in \CH^{p+1}(G)$.
Let $i=0,1,\dots, p-1$. The element $\bar h^i$ defines the following two morphisms of Chow motives:
\[
f_i: M(X)\to \Z((p+1)i),\quad\quad g_i:\Z((p+1)(p-1-i))\to M(X).
\]
Let
\[
R=\Z\oplus \Z(p+1)\oplus \Z(2p+2)\oplus\cdots\oplus\Z(p^2-1).
\]
We thus have the following two morphisms:
\[
f: M(X)\to R,\quad\quad g:R\to M(X).
\]
The composition $f\circ g$ is $c$ times the identity, where $c=\deg\bar h^{p-1}$. As $c$
is prime to $p$ by Theorem \ref{power}, switching to the {\em Chow motives with coefficients in $\Z_{(p)}$},
we have a decomposition
\begin{equation}\label{decom}
M(X)= R\oplus N
\end{equation}
for some motive $N$.

\section{The category of $D$-motives}\label{dmotives}

Let $D$ be a central simple algebra of prime degree $p$ over $F$.
For a field extension $L/F$, let $N^D_i(L)$ be the subgroup
of the Milnor $K$-group $K^M_i(L)$ generated by the norms from finite field extensions
of $L$ that split the algebra $D$.

Consider the
Rost cycle module (see \cite{Rost98a}):
\[
L\mapsto K^D_*(L):= K^M_*(L)/N^D_*(L),
\]
and the corresponding cohomology theory with the ``Chow groups"
\[
\CH_D^i(X):=A^i(X,K^D_i).
\]
Note that $\CH_D^i(X)=0$ if $D$ is split over $F(x)$ for all points $x\in X$.

Let $S=\SB(D)$ be the Severi-Brauer variety of right ideals of $D$ of dimension $p$. We have $\dim S=p-1$.

\begin{lemma}
\label{lemma}
For a variety $X$ over $F$, the group
$\CH_D(X)$
is naturally isomorphic to the cokernel of the push-forward
homomorphism
$\pr_*:\CH(X\times S)\ra \CH(X)$
given by the projection $\pr:X\times S\to X$.
\end{lemma}

\begin{proof}
The composition
\[
\CH(X\times S)\xra{\pr_*} \CH(X) \ra  \CH_D(X)
\]
factors through the trivial group $\CH_D(X\times S)$ and therefore, is zero. This defines
a surjective homomorphism
\[
\alpha: \Coker(\pr_*)\to \CH_D(X).
\]
The inverse map
is obtained by showing that the quotient map
$\CH(X)\to\Coker(\pr_*)$ factors through $\CH_D(X)$.

The kernel of the homomorphism $\CH(X)\to\CH_D(X)$ is generated by $[x]$ with $x\in X$ such that the algebra $D_{F(x)}$ is split and by $p[x]$ with arbitrary $x\in X$.
The fiber of $\pr$ over $x$ has a rational point $y$ in the first case and a degree $p$ closed point $y$ in the second. The generators are equal to $\pr_*([y])$ in both cases.
It follows that they vanish in $\Coker\pr_*$.
\end{proof}


Let $G=\gSL_1(D)$.

\begin{corollary}\label{isog}
The natural map $\CH^i(G)\to\CH_D^i(G)$ is an isomorphism for all $i>0$.
\end{corollary}

\begin{proof}
The algebra $D$ is split over $S$. More precisely, $D_X=\End_X(I^\vee)$ for the rank $p$
canonical vector bundle $I$ over $S$
(see \cite[Lemma 2.1.4]{MR2942107}).
By \cite[Theorem 4.2]{Suslin91a}, the pull-back homomorphism
$\CH^*(S)\to \CH^*(G\times S)$ is an isomorphism. Therefore, $\CH^j(G\times S)=0$ if $j>p-1=\dim(S)$.
\end{proof}

Let $X$ be a smooth compactification of $G$. Write $X^k=X\times X\times\dots\times X$ ($k$ times).

\begin{lemma}\label{restriction}
The restriction homomorphism
$\CH_D^*(X^k)\to \CH_D^*(G^k)$ is an isomorphism.
\end{lemma}

\begin{proof}
Let $Z=X^k\setminus G^k$. By Lemma \ref{infinity}, the residue field of every point in $Z$ splits $D$, hence
$\CH_D(Z)=0$. The statement follows from the exactness of
the localization sequence
\[
\CH_D(Z)\to \CH_D(X^k)\to\CH_D(G^k)\to 0. \qedhere
\]
\end{proof}

It follows from Lemma \ref{restriction} and Corollary \ref{isog} that $\CH_D^i(X)\simeq \CH^i(G)$ for $i>0$.

Consider the category of motives of smooth complete varieties over $F$  associated to the cohomology theory
$\CH^*_D(X)$
(see \cite{NZ06}).
Write $M^D(X)$ for the
motive of a smooth complete variety $X$. We call $M^D(X)$ the \emph{$D$-motive of $X$}.
Recall that the group of morphisms between $M^D(X)$ and $M^D(Y)$ for $Y$ of pure dimension $d$ is equal to
$\CH_D^d(X\times Y)$. Let $\Z^D$ the motive of the point $\Spec F$.

Recall that we write $M(X)$ for the usual Chow motive of $X$. We have a functor $N\mapsto N^D$ from the category of Chow motives to the category
of $D$-motives.

\begin{proposition}
\label{dirsum}
Let $N$ be a Chow motive. Then $N^D=0$ if and only if $N$ is isomorphic to a direct summand of $N\tens M(S)$.
\end{proposition}

\begin{proof}
As $M^D(S)=0$, we have $N^D=0$ if $N$ is isomorphic to a direct summand of $N\tens M(S)$.

Conversely, suppose $N^D=0$. Let $N=(X,\rho)$, where $X$ is a smooth complete variety of pure dimension $d$ and
$\rho\in\CH^d(X\times X)$ is a projector.
By Lemma \ref{lemma}, we have $\rho=f_*(\theta)$ for some
$\theta\in \CH^{d+p-1}(X\times (X \times S))$, where $f:X\times X\times S\ra X\times X$ is the projection. Then
$$
f_*\big((\rho\otimes\id_S)\circ\theta\circ\rho\big)=\rho
$$
and $(\rho\otimes\id_S)\circ\theta\circ\rho$ can be viewed as a morphism $N\to  N\tens M(S)$ splitting on the right the natural morphism $N\tens M(S)\to N$.
\end{proof}

The morphisms $f$ and $g$ in Section \ref{rost} give rise to the morphisms $f^D: M^D(X)\to R^D$ and $g^D: R^D \to M^D(X)$ of $D$-motives.

\begin{proposition}\label{isom}
The morphism $f^D: M^D(X)\to R^D$ is an isomorphism in the category of $D$-motives.
\end{proposition}

\begin{proof}
As $\CH^{p^2-1}_D(X\times X)\simeq \CH^{p^2-1}_D(G\times G)$ by Lemma \ref{restriction}, the composition $g^D\circ f^D$ is multiplication by
$c\in\Z$ from Proposition \ref{diagonal}.
By Theorem \ref{power}, $c$ is not divisible by $p$.
Finally, $p\CH_D(G\times G)=0$.
\end{proof}

If $D$ is a central division algebra, it follows from Proposition \ref{isom} and Corollary \ref{isog} that for every $i>0$,
\begin{equation}\label{chowchow}
\CH^i(G)= \CH^i_D(X)= \CH^i_D(R)=
\left\{
  \begin{array}{ll}
    (\Z/p\Z)h^j, & \hbox{if $i=(p+1)j\leq p^2-1$;} \\
    0, & \hbox{otherwise,}
  \end{array}
\right.
\end{equation}
where $h=\partial_G(q_G)$.

We can compute the Chow ring of $G$.

\begin{theorem}\label{chow2}
Let $D$ be a central division algebra of prime degree $p$, $G=\gSL_1(D)$ and $h=\partial_G(q_G)\in\CH^{p+1}(G)$. Then
  \newcommand{\xqedhere}[2]{%
  \rlap{\hbox to#1{\hfil\llap{\ensuremath{#2}}}}}
\[
\CH(G)=\Z\cdot 1\oplus (\Z/p\Z)h\oplus (\Z/p\Z)h^2\oplus\cdots\oplus (\Z/p\Z)h^{p-1}.
\]
\end{theorem}

\begin{proof}
If $F$ is a perfect field, $G$ admits a smooth compactification $X$ by Proposition \ref{torexist}. The statement follows from (\ref{chowchow}).
In general, we proceed as follows.

A variety $X$ over $F$ is called \emph{$D$-complete} is there is a compactification $\overline X$ of $X$
such that $D$ is split by the residue field of every point
in $\overline X\setminus X$.
Note that the restriction map $\CH(\overline X\times U)\to \CH(X\times U)$ is an isomorphism for every variety $U$.
By the proof of Lemma \ref{infinity}, $G$ is a $D$-complete variety.

We extend the category of $D$-motives by adding the motives $M^D(X)$ of smooth $D$-complete varieties $X$. If $X$ and $Y$ are two smooth $D$-complete
varieties with $Y$ equidimensional of dimension $d$, we define $\Hom(M^D(X),M^D(Y)):=\CH_D^d(X\times Y)$. The composition homomorphism
\[
\CH_D^d(X\times Y)\tens \CH_D^r(Y\times Z)\to \CH_D^r(X\times Z)
\]
is given by
\[
\alpha\tens\beta\mapsto p_{13 *}\( p_{12}^*(\alpha)\cdot p_{23}^*(\beta)\),
\]
where $p_{ij}$ are the three projections of $X\times Y \times Z$ on $X$, $Y$ and $Z$, and the push-forward map $p_{13 *}$ is defined as the composition
\[
p_{13 *}: \CH_D^{d+r}(X\times Y \times Z)\simeq \CH_D^{d+r}(X\times \overline Y \times Z)\to \CH_D^r(X\times Z).
\]
Here $\overline Y$ is a compactification of $Y$ satisfying the condition in the definition of a $D$-complete variety and the second map
is the push-forward homomorphism for the proper projection $X\times \overline Y \times Z \to X \times Z$.

By Proposition \ref{diagonal} and Theorem \ref{rostth}, the powers of $h=\partial_G(q_G)$ yield the following decomposition of $D$-motives
(with coefficients in $\Z_{(p)}$):
\[
M^D(G)\simeq \Z_{(p)}^D\oplus\Z_{(p)}^D(p+1)\oplus\cdots\oplus\Z_{(p)}^D(p^2-1).
\]
The result follows as $\CH^i(G)=\CH^i_D(G)$ for $i>0$ by Corollary \ref{isog}.
\end{proof}

\section{Motivic decomposition of compactifications of $\gSL_1(D)$}

Let $D$ be a central division $F$-algebra of degree a power of a prime $p$  and $S=\SB(D)$.
We work with motives with $\Z_{(p)}$-coefficients in this section.

\begin{proposition}
\label{smooth Y}
Let $X$ be a connected smooth complete variety over $F$ such that
the motive of $X$ is split over every splitting field of $D$ and $D$ is split over $F(X)$.
Then the motive of $X$ is a direct sum of shifts of the motive
of $S$.
\end{proposition}

\begin{proof}
Note that the variety $X$ is {\em generically split}, that is, its motive is split over $F(X)$.
In particular, $X$ satisfies the nilpotence principle, \cite[Proposition 3.1]{MR2393083}.
Therefore, it suffices to prove the result for motives with coefficients in $\F_p$:
any lifting of an isomorphism of the motives with coefficients in $\F_p$ to the coefficients $\Z_{(p)}$
will be an isomorphism since it will become an isomorphism over any splitting field of $D$.

For $\F_p$-coefficients, here is the argument.
%
The (isomorphism class of the) upper motive $U(X)$ is
well-defined and, by the arguments as in the proof of
\cite[Theorem 3.5]{upper},
the motive of $X$ is a sum of shifts of $U(X)$.
Besides, $U(X)\simeq U(S)$, cf. \cite[Corollary 2.15]{upper}.
Finally, $U(S)=M(S)$ because the motive of $S$ is indecomposable, \cite[Corollary 2.22]{upper}.
\end{proof}

From now on, the degree of the division algebra $D$ is $p$.
Recall that we work with motives with coefficients in $\Z_{(p)}$.
So, we set
\[
R=\Z_{(p)}\oplus \Z_{(p)}(p+1)\oplus \Z_{(p)}(2p+2)\oplus\cdots\oplus\Z_{(p)}(p^2-1)
\]
now.

\begin{theorem}
\label{main}
Let $F$ be a field, $D$ a central division $F$-algebra of prime degree $p$, $G=\gSL_1(D)$, $X$ a smooth compactification of $G$, and $M(X)$ its Chow motive with $\Z_{(p)}$-coefficients.
Assume that  $M(X)$ is split over every splitting field of $D$ (see Example \ref{twist}).
Then the motive $M(X)$ (over $F$) is isomorphic to the direct sum of $R$ and a direct sum of shifts of $M(S)$.
\end{theorem}

\begin{proof}
By (\ref{decom}),
$M(X)=R\oplus N$ for a motive $N$ and by Proposition \ref{isom}, $N^D=0$. It follows from Proposition
\ref{dirsum} that $N$ is isomorphic to a direct summand of $N\tens M(S)$.
In its turn, $N\tens M(S)$ is a direct summand of $M(X\times S)$.
In view of Proposition \ref{smooth Y}, $M(X\times S)$ is a direct sum of shifts of $M(S)$.
By the uniqueness
of the decomposition \cite[Corollary 35]{MR2264459}
and indecomposability of $M(S)$ \cite[Corollary 2.22]{upper}, the motive $N$ is a direct sum of shifts of $M(S)$.
\end{proof}



\begin{theorem}\label{maincor}
\label{cor1}
Let $E$ be an $\gSL_1(D)$-torsor and $X$ a smooth
compactification of $E$ such that the motive $M(X)$ is split over every splitting field of $D$ (see Example \ref{twist}).
Then $X$ satisfies the nilpotence principle.
Besides, the motive $M(X)$ is isomorphic to the direct sum of the Rost motive $\cR$ of $X$ and a direct sum of shifts of $M(S)$.
The above decomposition is the unique decomposition of $M(X)$ into a direct sum of indecomposable motives.
\end{theorem}

\begin{proof}
By saying that $X$ satisfies the nilpotence principle, we mean that it does it for any coefficient ring, or, equivalently,  for $\Z$-coefficients.
However, since the integral motive of $X$ is split over a field extension of degree $p$, it suffices to check that $X$ satisfies the nilpotence principle for $\Z_{(p)}$-coefficients, where we can simply refer to \cite[Theorem 92.4]{EKM} and Theorem \ref{main} (applied to $X$ over $F(X)$).

It follows that it suffices
to get the motivic decomposition of Theorem \ref{cor1} for $\Z_{(p)}$-coefficients replaced by
$\F_p$-coefficients.
For $\F_p$-coefficients
we use the following modification of \cite[Proposition 4.6]{hypernew}:

\begin{lemma}
\label{9.5}
Let $S$ be a geometrically irreducible variety with the motive satisfying the nilpotence principle and becoming split over an extension of the base field.
Let $M$ be a summand of the motive of some smooth complete variety $X$.
Assume that there exists a field extension $L/F$ and an integer $i\in\Z$ such that the change of field homomorphism
$\Ch(X_{F(S)})\to\Ch(X_{L(S)})$ is surjective and the motive $M(S)(i)_L$ is an indecomposable summand of $M_L$.
Then $M(S)(i)$ is an indecomposable summand of $M$.
\end{lemma}

\begin{proof}
It was assumed in \cite[Proposition 4.6]{hypernew} that the field extension $L(S)/F(S)$ is purely transcendental.
But this assumption was only used to ensure that the change of field homomorphism
$\Ch(X_{F(S)})\to\Ch(X_{L(S)})$ is surjective.
Therefore the old proof works.
\end{proof}

We apply Lemma \ref{9.5} to our $S$ and $X$ (with $L=F(X)$).
First we take $M=M(X)$ and
using Theorem \ref{main}, we   extract from $M(X)$ our first copy of shifted $M(S)$.
Then we apply Lemma \ref{9.5} again, taking for $M$ the complementary summand of$M(X)$.
Continuing this way, we eventually extract from $M(X)$ the same number of (shifted) copies of $M(S)$ as we have by Theorem \ref{main} over $F(X)$.
Let $\cR$ be the remaining summand of $M(X)$.
By uniqueness of decomposition, we have $\cR_{F(X)}\simeq R$ so that $\cR$ is the Rost motive.
It is indecomposable (over $F$), because the degree of every closed point on $X$ is divisible by $p$.

The uniqueness of the constructed decomposition follows by \cite[Theorem 3.6 of Chapter I]{MR0249491}, because the endomorphism rings of $M(S)$ and of $\cR$ are local (see \cite[Lemma 3.3]{snv}).
\end{proof}

\begin{remark}
If $X$ is an equivariant toroidal compactification of $\gSL_1(D)$, the number of motives $M(S)$ in the decomposition of Theorem \ref{maincor} is equal to $s(p-1)!-1$, where $s$ is the number of cones of maximal dimension in the fan of the associated toric variety (see Theorem \ref{semis}).
\end{remark}

\begin{example}\label{semenovex}
Let $X$ be the (non-toroidal) equivariant compactification of $\gSL_1(D)$ with $p=3$ considered in \cite{MR2400993}. Since $P_X(t)=t^8 + t^7 + 2t^6 + 3t^5 + 4t^4 + 3t^3 + 2t^2 + t þ + 1$, we have \[
M(X)\simeq \cR\oplus M(S)(1)\oplus M(S)(2)\oplus M(S)(3)\oplus M(S)(4)\oplus M(S)(5).
\]
\end{example}

\begin{example}
Let $X$ be the toroidal equivariant compactification of $\gSL_1(D)$ with $p=3$ considered in Example \ref{sl2} in the split case. We have
\[
M(X)\simeq \cR\oplus M(S)(1)^{\oplus 3}\oplus M(S)(2)^{\oplus 5}\oplus M(S)(3)^{\oplus 7}\oplus M(S)(4)^{\oplus 5}\oplus M(S)(5)^{\oplus 3}.
\]
\end{example}

\begin{corollary}
\label{cor2}
Let $E$ be a nonsplit $\gSL_1(D)$-torsor.
Assume that $\Char F=0$.
Then $\CH(E)=\Z$.
\end{corollary}

\begin{proof}
Since $p\CH^{>0}(E)=0$, it suffices to prove that $\CH^{>0}(E)=0$ for $\Z$-coefficients replaced by $\Z_{(p)}$-coefficients.
Below $\CH$ stands for Chow group with $\Z_{(p)}$-coefficients.

We prove that $\CH(E)=\CH_D(E)$ by the argument of Corollary \ref{isog}.
It remains to show that $\CH^{>0}_D(E)=0$.

Let $X$ be a compactification of $E$ as in Theorem \ref{cor1}.
Since $\CH_D(X)$ surjects onto $\CH_D(E)$ and $\CH_D(S)=0$, it suffices to check that
$\CH^{>0}_D(\cR)=0$.
Actually, we have $\CH_D(\cR)\simeq\CH_D(E)$ (see Section \ref{dmotives}). Moreover, the $D$-motive of $\cR$ is isomorphic to $M^D(E)$.

The Chow group $\CH^{>0}(\cR)$ has been computed in \cite[Appendix RM]{snv}
(the characteristic assumption is needed here).
The generators of the torsion part, provided in \cite[Proposition SC.21]{snv}, vanish in $\CH_D(\cR)$ by construction.
The remaining generators are norms from a degree $p$ splitting field of $D$ so that they vanish in $\CH_D(\cR)$, too.
Hence $\CH^{>0}_D(\cR)=0$ as required.
\end{proof}


\end{document}